\documentclass[a4paper,twoside]{article}
\usepackage{a4}
\usepackage{amssymb}
\usepackage{amsmath}
\usepackage{upref}
\usepackage[active]{srcltx}
\usepackage[colorlinks,citecolor=blue,linkcolor=blue]{hyperref}
\usepackage[dvipsnames]{color}
\allowdisplaybreaks[2] 
\newcount\minutes \newcount\hours
\hours=\time
\divide\hours 60
\minutes=\hours
\multiply\minutes -60
\advance\minutes \time
\newcommand{\klockan}{\the\hours:{\ifnum\minutes<10 0\fi}\the\minutes}
\newcommand{\tid}{\today\ \klockan}
\newcommand{\prtid}{\smash{\raise 10mm \hbox{\LaTeX ed \tid}}}
\renewcommand{\prtid}{}
\makeatletter
\def\sectionmark#1{} 
\def\subsectionmark#1{}
\newcommand{\sectnr}{\ifnum \c@secnumdepth >\z@
                 \thesection.\hskip 1em\relax \fi}
\def\@evenhead{\footnotesize\rm\thepage\hfil\leftmark\hfil\llap{\prtid}}
\def\@oddhead{\footnotesize\rm\rlap{\prtid}\hfil\rightmark\hfil\thepage}
\def\tableofcontents{\section*{Contents} 
 \@starttoc{toc}}
\makeatother
\makeatletter
\def\@biblabel#1{#1.}
\makeatother
\makeatletter
\let\Thebibliography=\thebibliography
\renewcommand{\thebibliography}[1]{\def\@mkboth##1##2{}\Thebibliography{#1}
\addcontentsline{toc}{section}{References}
\frenchspacing 
\setlength{\@topsep}{0pt}
\setlength{\itemsep}{0pt}%
\setlength{\parskip}{0pt plus 2pt}%
}
\makeatother
\makeatletter
\def\mdots@{\mathinner.\nonscript\!.%
 \ifx\next,.\else\ifx\next;.\else\ifx\next..\else
 \nonscript\!\mathinner.\fi\fi\fi}
\let\ldots\mdots@
\let\cdots\mdots@
\let\dotso\mdots@
\let\dotsb\mdots@
\let\dotsm\mdots@
\let\dotsc\mdots@
\def\vdots{\vbox{\baselineskip2.8\p@ \lineskiplimit\z@
    \kern6\p@\hbox{.}\hbox{.}\hbox{.}\kern3\p@}}
\def\ddots{\mathinner{\mkern1mu\raise8.6\p@\vbox{\kern7\p@\hbox{.}}%
    \raise5.8\p@\hbox{.}\raise3\p@\hbox{.}\mkern1mu}}
\makeatother
\makeatletter
\let\Enumerate=\enumerate
\renewcommand{\enumerate}{\Enumerate%
\setlength{\@topsep}{0pt}
\setlength{\itemsep}{0pt}%
\setlength{\parskip}{0pt plus 1pt}%
\renewcommand{\theenumi}{\textup{(\alph{enumi})}}%
\renewcommand{\labelenumi}{\theenumi}%
}
\let\endEnumerate=\endenumerate
\renewcommand{\endenumerate}{\endEnumerate\unskip}
\newcounter{saveenumi}
\makeatother
\makeatletter
\def\@seccntformat#1{\csname the#1\endcsname.\quad}
\makeatother
\newcommand{\authortitle}[2]{\author{#1}\title{#2}\markboth{#1}{#2}}
\newcommand{\art}[6]{{\sc #1, \rm #2, \it #3\/ \bf #4 \rm (#5), \mbox{#6}.}}

\newcommand{\auth}[2]{{#2. #1}}
\newcommand{\book}[3]{{\sc #1, \it #2, \rm #3.}}
\newcommand{\AND}{{\rm and }}
\RequirePackage{amsthm}
\newtheoremstyle{descriptive}%
  {\topsep}   
  {\topsep}   
  {\rmfamily} 
  {}          
  {\bfseries} 
  {.}         
  { }         
  {}          
\newtheoremstyle{propositional}%
  {\topsep}   
  {\topsep}   
  {\itshape}  
  {}          
  {\bfseries} 
  {.}         
  { }         
  {}          
\theoremstyle{propositional}
\newtheorem{thm}{Theorem}[section]
\newtheorem{theo}[thm]{Theorem}   
\newtheorem{prop}[thm]{Proposition}
\newtheorem{lem}[thm]{Lemma}
\theoremstyle{descriptive}
\newtheorem{definition}[thm]{Definition}

\newtheorem{remark}[thm]{Remark}

\makeatletter
\renewenvironment{proof}[1][\proofname]{\par
  \pushQED{\qed}%
  \normalfont
  \trivlist
  \item[\hskip\labelsep
        \itshape
    #1\@addpunct{.}]\ignorespaces
}{%
  \popQED\endtrivlist\@endpefalse
}
\makeatother
{\catcode`p =12 \catcode`t =12 \gdef\eeaa#1pt{#1}}      
\def\accentadjtext#1{\setbox0\hbox{$#1$}\kern   
                \expandafter\eeaa\the\fontdimen1\textfont1 \ht0 }
\def\accentadjscript#1{\setbox0\hbox{$#1$}\kern 
                \expandafter\eeaa\the\fontdimen1\scriptfont1 \ht0 }
\def\accentadjscriptscript#1{\setbox0\hbox{$#1$}\kern   
                \expandafter\eeaa\the\fontdimen1\scriptscriptfont1 \ht0 }
\def\accentadjtextback#1{\setbox0\hbox{$#1$}\kern       
                -\expandafter\eeaa\the\fontdimen1\textfont1 \ht0 }
\def\accentadjscriptback#1{\setbox0\hbox{$#1$}\kern     
                -\expandafter\eeaa\the\fontdimen1\scriptfont1 \ht0 }
\def\accentadjscriptscriptback#1{\setbox0\hbox{$#1$}\kern 
                -\expandafter\eeaa\the\fontdimen1\scriptscriptfont1 \ht0 }
\def\itoverline#1{{\mathsurround0pt\mathchoice
        {\rlap{$\accentadjtext{\displaystyle #1}
                \accentadjtext{\vrule height1.593pt}
                \overline{\phantom{\displaystyle #1}
                \accentadjtextback{\displaystyle #1}}$}{#1}}
        {\rlap{$\accentadjtext{\textstyle #1}
                \accentadjtext{\vrule height1.593pt}
                \overline{\phantom{\textstyle #1}
                \accentadjtextback{\textstyle #1}}$}{#1}}
        {\rlap{$\accentadjscript{\scriptstyle #1}
                \accentadjscript{\vrule height1.593pt}
                \overline{\phantom{\scriptstyle #1}
                \accentadjscriptback{\scriptstyle #1}}$}{#1}}
        {\rlap{$\accentadjscriptscript{\scriptscriptstyle #1}
                \accentadjscriptscript{\vrule height1.593pt}
                \overline{\phantom{\scriptscriptstyle #1}
                \accentadjscriptscriptback{\scriptscriptstyle #1}}$}{#1}}}}
\def\itunderline#1{{\mathsurround0pt\mathchoice
        {\rlap{$\underline{\phantom{\displaystyle #1}
                \accentadjtextback{\displaystyle #1}}$}{#1}}
        {\rlap{$\underline{\phantom{\textstyle #1}
                \accentadjtextback{\textstyle #1}}$}{#1}}
        {\rlap{$\underline{\phantom{\scriptstyle #1}
                \accentadjscriptback{\scriptstyle #1}}$}{#1}}
        {\rlap{$\underline{\phantom{\scriptscriptstyle #1}
                \accentadjscriptscriptback{\scriptscriptstyle #1}}$}{#1}}}}
\newcommand{\limplus}{{\mathchoice{\vcenter{\hbox{$\scriptstyle +$}}}
  {\vcenter{\hbox{$\scriptstyle +$}}}
  {\vcenter{\hbox{$\scriptscriptstyle +$}}}
  {\vcenter{\hbox{$\scriptscriptstyle +$}}}
}}
\newcommand{\limminus}{{\mathchoice{\vcenter{\hbox{$\scriptstyle -$}}}
  {\vcenter{\hbox{$\scriptstyle -$}}}
  {\vcenter{\hbox{$\scriptscriptstyle -$}}}
  {\vcenter{\hbox{$\scriptscriptstyle -$}}}
}}
\newdimen\extrawidth
\def\iintlim#1#2{\setbox0\hbox{$\scriptstyle#1$}%
	\setbox1\hbox{$\scriptstyle#2$}%
	\extrawidth=\wd1 \advance\extrawidth-\wd0
	\ifdim\extrawidth<0pt \extrawidth=0pt\fi%
	\int_{#1\kern\extrawidth \kern .5em}^{#2\kern -\wd1} \kern -.5em%
}
\numberwithin{equation}{section}
\newenvironment{ack}{\medskip{\it Acknowledgement.}}{}
\DeclareMathOperator{\Div}{div}
\renewcommand{\phi}{\varphi}
\newcommand{\eps}{\varepsilon}
\newcommand{\lm}{\lambda}
\newcommand{\al}{\alpha}
\newcommand{\de}{\delta}

\newcommand{\R}{\mathbf{R}}
\newcommand{\Rn}{\mathbf{R}^n}
\newcommand{\Rno}{\mathbf{R}^{n+1}}

\newcommand{\Thetat}{\widetilde{\Theta}}
\newcommand{\Thetap}{\Theta'}
\newcommand{\clThetap}{{\overline{\Theta}\mspace{1mu}}'}
\newcommand{\Thetah}{\widehat{\Theta}}
\newcommand{\ut}{\tilde{u}}
\newcommand{\ft}{\tilde{f}}

\newcommand{\p}{{$p\mspace{1mu}$}}

\newcommand{\uP}{\itoverline{H}}
\newcommand{\uPa}{{\itoverline{H}\mspace{1mu}}^a}
\newcommand{\lP}{\itunderline{H}}
\newcommand{\uPind}[1]{\itoverline{H}_{#1}}
\newcommand{\lPind}[1]{\itunderline{H}_{#1}}
\newcommand{\UU}{\mathcal{U}}%
\newcommand{\LL}{\mathcal{L}}%
\newcommand{\bdy}{\partial}
\newcommand{\bdry}{\partial}
\newcommand{\grad}{\nabla}
\newcommand{\setm}{\setminus}
\newcommand{\alp}{\alpha}
\newcommand{\ga}{\gamma}
\newcommand{\la}{\lambda}
\newcommand{\td}{{\tilde{\delta}}}
\newcommand{\be}{\beta}
\begin{document}
%
%
\authortitle{Anders Bj\"orn, Jana Bj\"orn and Ugo Gianazza}
{The Petrovski\u{\i} criterion and barriers
for degenerate and singular \p-parabolic equations}
\author{
Anders Bj\"orn \\
\it\small Department of Mathematics, Link\"oping University, \\
\it\small SE-581 83 Link\"oping, Sweden\/{\rm ;}
\it \small anders.bjorn@liu.se
\\
\\
Jana Bj\"orn \\
\it\small Department of Mathematics, Link\"oping University, \\
\it\small SE-581 83 Link\"oping, Sweden\/{\rm ;}
\it \small jana.bjorn@liu.se
\\
\\
Ugo Gianazza \\
\it\small Department of Mathematics ``F. Casorati'', Universit\`a di Pavia,\\
\it\small via Ferrata 1, 27100 Pavia, Italy\/{\rm ;}
\it\small gianazza@imati.cnr.it
}

\date{}
\maketitle

\noindent{\small
{\bf Abstract}.
In this paper we obtain sharp Petrovski\u\i\ criteria for the \p-parabolic
equation, both in the degenerate case $p>2$ and the singular case $1<p<2$.
We also give an example of an irregular boundary point at which there
is a barrier, thus showing that regularity cannot be characterized
by the existence of just one barrier.
}

\bigskip
\noindent
{\small \emph{Key words and phrases}:
barrier, degenerate parabolic,
Perron's method, Petrovski\u\i\   criterion,
\p-parabolic equation,
regular boundary point,
singular parabolic.
}


\section{Introduction}

In~\cite{Petro2} Petrovski\u\i\ proved the following result.

\medskip

\noindent{\bf Petrovski\u\i's criterion.} 
\emph{The origin $(0,0)$ is regular for the heat equation
$\bdry_t u-\Delta u=0$ in $\R^{n+1}$ 
with respect to the domain
\begin{equation} \label{eq-Pet-p=2}
 \{(x,t) \in \R^{n+1}: 
        |x|<K\sqrt{-t}\sqrt{\log |{\log(-t)}|} \text{ and } -1< t<0\}
\end{equation}
if and only if $K\le2$.}

\medskip

In this paper we obtain similar results for the nonlinear \p-parabolic equation
\begin{equation} \label{eq:para}
\partial_t u- \Delta_p u
    :=\frac{\partial u}{\partial t}- \Div(|\nabla u|^{p-2} \nabla u)
    =0
\end{equation}
both in the \emph{degenerate} case $p>2$ and the \emph{singular} case $1<p<2$.
For $p=2$, \eqref{eq:para}  reduces to the usual heat equation.
(The gradient $\grad u$ and the \emph{\p-Laplacian} $\Delta_p$
are taken with respect to $x\in \R^n$.)

Boundary regularity for the \p-parabolic equation has been studied
by Lindqvist~\cite{lindqvist95}, Kilpel\"ainen--Lindqvist~\cite{KiLi96}
and Bj\"orn--Bj\"orn--Gianazza--Par\-vi\-ain\-en~\cite{BBGP}.
Sufficient Petrovski\u{\i}-type conditions were given in \cite{lindqvist95}
and \cite{BBGP}.
Boundary regularity has also been studied for the 
normalized \p-parabolic equation 
$\partial_t u- |\nabla u|^{2-p}\Delta_p u$=0 by
Banerjee--Garofalo~\cite{BanerjeeGarofalo}.

There are some significant differences in the theory of boundary regularity
for $p \ne 2$ and for the heat equation ($p=2$).
The scaling argument
in Section~\ref{sect-scaling}
shows that for $p \ne 2$ one cannot have a Petrovski\u\i-type
criterion where 
a parameter similar to  $K$ in \eqref{eq-Pet-p=2}
dictates regularity.
Instead we obtain the following result.
See also Remark~\ref{rmk-complete-reg}.

\begin{thm} \label{thm-Petrovskii}
\textup{(Petrovski\u\i-type criteria for $1<p<\infty$)}
Let $K>0$, $q>0$  and 
\[
\Theta=\{(x,t)\in\Rno: |x|<K(-t)^{q} \text{ and }  -1 < t<0\}.
\]
Then the following are true\/\textup{:}
\begin{enumerate}
\item
If $p>2$, then $(0,0)$ is regular if and only if $q>1/p$.
\item 
If $p=2$, then $(0,0)$ is regular if and only if $q \ge 1/2$.
\item 
If $1<p<2$, then $(0,0)$ is regular if $q > 1/p$ and irregular if $q<1/p$.
\end{enumerate}
In all cases regularity is with respect to $\Theta$.
\end{thm}

For $p=2$ this follows quite directly from the Petrovski\u\i\  criterion
above.
For $p<2$ and $q=1/p$ we do not know whether $(0,0)$ is regular or not,
but the case $p=2$ shows that it is quite possible that $p=2$ 
is a break point for this result and that $(0,0)$ may be regular
when $p<2$ and $q=1/p$.

The Petrovski\u\i-type criterion in 
Lindqvist~\cite[Theorem, p.~571]{lindqvist95}
and Bj\"orn--Bj\"orn--Gianazza--Parviainen~\cite[Theorem~6.1]{BBGP}
gives regularity when 
\[
    p>2 
    \quad \text{and} \quad
    q \ge \frac{1}{p}+\frac{n(p-2)^2}{\la p},
\]
where from now on we use the shorthand $\la=n(p-2)+p$.
It was conjectured in~\cite[p.\ 572]{lindqvist95}
that this would be sharp, 
which is now disproved by Theorem~\ref{thm-Petrovskii}.
For $p<2$ and $q>1/p$ regularity follows from 
Proposition~7.1 in \cite{BBGP} 
(and 
\cite[Proposition~3.4]{BBGP} when $0 < K \le 1$).

In Kilpel\"ainen--Lindqvist~\cite[pp.\ 676--677]{KiLi96}
 it was shown that $(0,0)$ is an irregular boundary point with 
respect to the so-called Barenblatt balls when $p>2$, i.e.\
for $q=1/\la<1/p$, with $K$ dependent on $p$.
Lindqvist~\cite[footnote p.\ 572]{lindqvist95} 
also states 
that ``it is not too
difficult to show'' irregularity for $q =1/p$ when $p>2$. 
Theorem~\ref{thm-Petrovskii} extends these results and completes the picture.
As a matter of fact, 
for $p>2$ we provide more powerful criteria in Theorem~\ref{thm:Petr-deg}
and Proposition~\ref{Prop:Petr-deg-irr}.
As we do not know what happens when $p<2$ and $q=1/p$, we
have refrained from giving such criteria when $p<2$.

We are also interested in barrier characterizations.
Already Kil\-pe\-l\"ai\-nen--Lind\-qvist~\cite{KiLi96} suggested
that regularity can be characterized using 
one (traditional) 
barrier.
Such a criterion
turned out to be problematic, 
and it
has been an open problem since then 
whether a single (traditional)
barrier guarantees regularity.
A criterion using a family of barriers was obtained 
in~\cite[Theorem~3.3]{BBGP}.
In this paper we prove
the following result.

\begin{prop} \label{prop-one-barrier-intro}
Let $1<p<2$, $K>0$ and $0<q<1/p$.
Then there is a traditional barrier at $(0,0)$ for the domain
\[
\Theta=\{(x,t)\in\Rno: |x|<K(-t)^{q} \text{ and }  -1 < t<0\}
\]
despite the fact that $(0,0)$ is irregular.
\end{prop}

This shows that regularity cannot be characterized using only one barrier, at least
not for $p<2$. We conjecture
that this is true also for $p>2$, but we have not
been able to find a counterexample in the degenerate range.

We end this introduction by mentioning that quite a lot
of attention has been given to the study of nonlinear parabolic
problems in the last 20--30 years, in particular for 
the \p-parabolic equation as here.
See, for example, 
B\"ogelein--Duzaar--Mingione~\cite{BDM},
DiBenedetto~\cite{dibe-mono},
DiBenedetto--Gianazza--Vespri~\cite{DBGV-mono}, 
Kuusi--Mingione~\cite{KuusiMing}
and Bj\"orn--Bj\"orn--Gianazza--Parviainen~\cite{BBGP}
for the recent history and many more references to the current literature.

\begin{ack}
The  first two authors were supported by the Swedish Research Council.
Part of this research was done while J.~B. visited
Universit\`a di Pavia in 2014,
and the paper was completed while U.~G. visited
Link\"oping University in 2015.
They want to thank these departments for the hospitality.
\end{ack}

\section{Preliminaries}

We will use the notation and several 
results from 
Bj\"orn--Bj\"orn--Gianazza--Par\-vi\-ain\-en~\cite{BBGP}.
Here we will be brief and only introduce and discuss what we really need,
see \cite{BBGP} for a more extensive discussion.

\medskip

\emph{From now on we will always assume that 
$\Theta\subset\R^{n+1}$ is a nonempty bounded open set
and $1<p< \infty$.}

\medskip

 Let $U$ be an
open set in $\Rn$.  The \emph{parabolic boundary} of the cylinder
$U_{t_1,t_2}:=U\times (t_1,t_2)\subset \R^{n+1}$ is
\[
\partial_p U_{t_1,t_2}=(\overline U\times \{t_1\})\cup(\partial U\times (t_1,t_2]).
\]

By the \emph{parabolic Sobolev space} $L^p(t_1,t_2;W^{1,p}(U))$,
with $t_1<t_2$, we mean the space of functions $u(x,t)$ such that the mapping $x
\mapsto u(x,t)$ belongs to $W^{1,p}(U)$ for almost every $t_1 < t <
t_2$ and the norm
\[
\biggl(\iintlim{t_1}{t_2}\int_U (|u(x,t)|^{p} + |\nabla
u(x,t)|^{p})\,dx\,dt\biggr)^{1/p}
\]
is finite. The definition of the space
$L^p(t_1,t_2;W_{0}^{1,p}(U))$ is similar.
Analogously, by the space $C([t_1,t_2];L^p(U))$,
with $t_1<t_2$, we mean the space of functions $u(x,t)$, such that the mapping
$t\mapsto\int_U|u(x,t)|^p\,dx$ is continuous in the time interval $[t_1,t_2]$.
(The gradient $\nabla$ and divergence $\Div$ are always taken with
respect to the $x$-variables in this paper.)
We can now introduce the notion of weak solution.

\begin{definition}
A function $u:\Theta \to [-\infty,\infty]$ is a  
\emph{weak solution} to equation \eqref{eq:para} 
if whenever $U_{t_1,t_2} \Subset \Theta$ is an open cylinder, 
we have 
$u \in C([t_1,t_2];L^2(U))\cap L^{p}(t_1,t_2;W^{1,p}(U))$, and  
 $u$ satisfies the integral equality
\[ 
\iintlim{t_1}{t_2}\int_{U} |\nabla u|^{p-2} \nabla u \cdot
\nabla\phi \, dx\,dt - \iintlim{t_1}{t_2}\int_{U} u
\frac{\partial\phi}{\partial t} \, dx\,dt   =  0
\quad \text{for all }\phi \in C_0^\infty(U_{t_1,t_2}).
\] 
A \emph{\p-parabolic function} is a continuous weak solution.

A function $u$
is a  \emph{weak supersolution}
if whenever $U_{t_1,t_2} \Subset
\Theta$ we have $u \in
L^{p}(t_1,t_2;W^{1,p}(U))$ and the 
left-hand side 
above is nonnegative for all
nonnegative $\phi \in C_0^\infty(U_{t_1,t_2})$.
For simplicity, we will 
omit \emph{weak}, when talking of weak supersolutions. 
\end{definition}

The most important \p-parabolic function is the \emph{Barenblatt solution} 
\cite{Bare52} 
${\mathcal B_p}:{\mathbf R}^n\times(0,\infty)\to[0,\infty)$ defined by
\[
\mathcal B_p(x,t)=t^{-n/\lambda}\biggl(C-\frac{p-2}{p}\lambda^{1/(1-p)}
\biggl(\frac{|x|}{t^{1/\lambda}}\biggr)^{p/(p-1)}\biggr)_\limplus^{(p-1)/(p-2)},\quad\lambda=n(p-2)+p,
\]
where $C>0$ is an arbitrary constant.
Even though it was introduced in the context of degenerate equations for $p>2$,
it is well defined also for $p<2$, provided that $\lambda>0$,
i.e.\ that $2n/(n+1)<p<2$.
We will not directly use the Barenblatt solution in this paper, but 
some of our expressions are closely related to the
Barenblatt solution.

\begin{definition}\label{def:superparabolic}
A function $u:\Theta\rightarrow (-\infty,\infty]$
is \emph{\p-super\-parabolic} if
\begin{enumerate}
\renewcommand{\theenumi}{\textup{(\roman{enumi})}}%
\item $u$ is lower semicontinuous;
\item $u$ is finite in a dense subset of ${\Theta}$;
\item $u$ satisfies the following comparison principle on each space-time
  box $Q_{t_1,t_2}\Subset{\Theta}$: If $h$ is \p-parabolic in
  $Q_{t_1,t_2}$ and continuous on $\overline{Q}_{t_1,t_2}$, and if
  $h\leq u$ on $\partial_p Q_{t_1,t_2}$, then $h\leq u$ in the whole $Q_{t_1,t_2}$.
\end{enumerate}

A function $v:\Theta\rightarrow [-\infty,\infty)$
is \emph{\p-sub\-parabolic} if $-u$ is \p-superparabolic.
\end{definition}

Here $Q_{t_1,t_2}$ is a \emph{space-time box} if it is of the form
$ Q_{t_1,t_2}=Q\times(t_1,t_2)$, where $Q=(a_1,b_1)\times\ldots\times(a_n,b_n)$.

The connection between \p-superparabolic functions and supersolutions 
is a delicate issue.
However, a continuous supersolution is \p-superparabolic
by the comparison principle 
of Korte--Kuusi--Parviainen~\cite[Lemma~3.5]{KoKuPa10}.

We will need the following parabolic comparison principle.

\begin{theo} \label{thm-parabolic-comparison}
\textup{(Parabolic comparison principle, \cite[Theorem~2.4]{BBGP})}
Suppose that $u$ is \p-super\-parabolic and $v$ is \p-sub\-parabolic
in $\Theta$.
Let $T \in \R$ 
and assume that
 \[
  \infty \ne    \limsup_{\Theta \ni (y,s)\rightarrow (x,t)} v(y,s)\leq
   \liminf_{\Theta \ni (y,s)\rightarrow (x,t)} u(y,s) \ne -\infty
 \]
for all
$(x,t) \in\{(x,t) \in \partial\Theta : t< T\}$.
Then $v\leq u$ in $\{(x,t) \in \Theta : t<T\}$.
\end{theo}

We now turn to the Perron method.
For us it will be enough to consider Perron solutions for
bounded functions, so for simplicity we restrict ourselves
to this case.

\begin{definition}   \label{def-Perron}
Given a bounded function $f \colon \bdy \Theta \to \R$,
let the upper class $\UU_f(\Theta)$ be the set of all
\p-superparabolic  functions $u$ on $\Theta$ which are
bounded below 
and such that
\[
    \liminf_{\Theta \ni \eta \to \xi} u(\eta) \ge f(\xi) \quad \text{for all }
    \xi \in \bdy \Theta.
\]
Define the \emph{upper Perron solution} of $f$  by
\[
    \uPind{\Theta} f (\xi) = \inf_{u \in \UU_f(\Theta)}  u(\xi), \quad \xi \in \Theta.
\]
Similarly, let the lower class 
$\LL_f(\Theta)$ be the set of all
\p-subparabolic functions $u$ on $\Theta$ which are
bounded above and such that
\[
\limsup_{\Theta \ni \eta \to \xi} u(\eta) \le f(\xi) \quad \text{for all }
\xi \in \bdy \Theta,
\]
and define the \emph{lower Perron solution}  of $f$ by
\[
    \lPind{\Theta} f (\xi) = \sup_{u \in \LL_f(\Theta)}  u(\xi), \quad \xi \in \Theta.
\]
\end{definition}

If the domain under consideration is clear 
from the context,
we will often drop $\Theta$ from the
notation above.
It follows from the 
parabolic comparison principle (Theorem~\ref{thm-parabolic-comparison})
that $\lP f \le \uP f$.
Moreover $\lP f = - \uP (-f)$.
Kilpel\"ainen--Lindqvist~\cite[Theorem~5.1]{KiLi96}
showed that both $\lP f$ and $\uP f$
are \p-parabolic.

The following simple lemma is easily proved by direct calculation.

\begin{lem}   \label{lem-Delta-powers}
For any $\al, C\in\R$ we have
\[
\Delta_p (C|x|^\al) = C\al |C\al|^{p-2} (n+(\al-1)(p-1)-1) |x|^{(\al-1)(p-1)-1}.
\]
In particular, if $\al=p/(p-2)$ 
and $C\al>0$ then 
\[
\Delta_p (C|x|^\al) = (C\al)^{p-1} (n+\al)|x|^\al 
= \frac{(C\al)^{p-1} \la}{p-2}|x|^\al,
\]
and if $\al=p/(p-1)$ 
and $C>0$ then $\Delta_p (C|x|^\al) = (C\al)^{p-1}n$.
\end{lem}

\section{Boundary regularity}\label{S:Boundary}

\begin{definition}
\label{def:regular}
A boundary point $\xi_0\in \partial \Theta$ is 
\emph{regular} with respect to $\Theta$, if
\[
        \lim_{\Theta \ni \xi \to \xi_0} \uP f(\xi)=f(\xi_0)
\]
whenever $f: \partial \Theta \to \R$ is continuous.
\end{definition}
 
Observe that since $\lP f = - \uP (-f)$, regularity can equivalently
be formulated using lower Perron solutions.

\begin{definition}
\label{def-barrier} 
Let $\xi_0\in \partial \Theta$. A family of functions 
$w_j: \Theta\to (0,\infty]$, $j=1,2,\ldots$, 
 is a \emph{barrier family} in $\Theta$ at the point  $\xi_0$ if
for each $j$, 
\begin{enumerate}
 \item\label{cond-first}  $w_j$ is a  positive \p-superparabolic function in $\Theta$;
 \item\label{cond-third} $\lim_{{\Theta \ni} \zeta\to \xi_0} w_j(\zeta)=0$;
 \item\label{cond-second-weak} 
for each $k=1,2,\ldots$, there is a $j$ such that 
\[
   \liminf_{\Theta \ni \zeta\to\xi} w_j(\zeta) \ge k
   \quad \text{for all } \xi \in \bdy \Theta 
   \text{ with }|\xi-\xi_0| \ge 1/k.
\]
\setcounter{saveenumi}{\value{enumi}}
\end{enumerate}
\bigskip
%

We also say that the family $w_j$ 
is a  \emph{strong barrier family} in $\Theta$ at the point  $\xi_0$ if,
in addition, the following conditions hold:
\begin{enumerate}
\setcounter{enumi}{\value{saveenumi}}
\item\label{cond-cont}  $w_j$ is continuous in $\Theta$;
\item\label{cond-second-strong} 
there is a nonnegative function $d \in C(\overline{\Theta})$, with
$d(z)=0$ if and only if $z=\xi_0$,
such that 
for each $k=1,2,\ldots$, there is a $j=j(k)$ such that 
$w_j \ge k d$ in $\Theta$.
\end{enumerate}
\end{definition}

\begin{thm}\label{thm:barrier-char}
\textup{(\cite[Theorem~3.3]{BBGP})} 
Let $\xi_0 \in \bdy \Theta$. Then the following are equivalent:
\begin{enumerate}
\renewcommand{\theenumi}{\textup{(\arabic{enumi})}}%
\item \label{i-reg}
$\xi_0$ is regular\/\textup{;}
\item \label{i-weak}
there is a barrier family at $\xi_0$\/\textup{;}
\item \label{i-strong}
there is a strong barrier family at $\xi_0$.
\end{enumerate}
\end{thm}

In classical potential theory,  
a \emph{barrier} is a superharmonic (when dealing with the 
Laplace equation) or superparabolic (when dealing with the heat equation) 
function $w$ 
such that 
\[ 
\lim_{\zeta\to \xi_0} w(\zeta)=0 
   \quad \text{and} \quad
  \liminf_{\zeta\to \xi} w(\zeta)>0
   \text{ for } \xi \in \bdy \Theta \setm \{\xi_0\}.
\]
Existence of such a \emph{single} barrier implies the regularity of a 
boundary point in these classical cases, 
since  one can scale and lift the barrier
(i.e.\ if $u$ is a barrier, then also $au+b$ is 
superharmonic/superparabolic,
where $a>0$ and $b \in \R$). 
A similar property holds also for the nonlinear \p-Laplace equation
$\Delta_p u=0$.
However, this is not the case 
for the \p-parabolic equation,
since it is not homogeneous: 
If $u$ is a supersolution, then 
$au$ (with $a>0$) is usually not a supersolution, 
even though, $u+a$ is indeed still a supersolution.

We 
say that $u$ is a
\emph{traditional barrier}  at $\xi_0 \in \bdy \Theta$  if

\begin{enumerate}
 \item  $u$ is a  positive \p-superparabolic function in $\Theta$;
 \item  $\lim_{{\Theta \ni} \zeta\to \xi_0} u (\zeta)=0$;
 \item $  \liminf_{\zeta\to \xi} u(\zeta)>0$
    for  all $ \xi \in \bdy \Theta \setm \{\xi_0\}$.
\end{enumerate}

It is clear that regularity implies the existence of a traditional
barrier (this follows e.g.\ from \ref{i-strong} in Theorem~\ref{thm:barrier-char}
above).
Conversely, as mentioned in the introduction,
it has been an 
open problem whether the existence
of a traditional barrier characterizes regularity, 
which we solve in the negative when $p<2$.

The following results are important consequences of the barrier
characterization in Theorem~\ref{thm:barrier-char}.

\begin{prop} \label{prop-restrict}
\textup{(\cite[Proposition~3.4]{BBGP})}
Let $\xi_0 \in \bdy \Theta$ and let $G \subset \Theta$ be open and such
that $\xi_0 \in \bdy G$.
If $\xi_0$ is regular with respect to $\Theta$,
then $\xi_0$ is regular with respect to $G$.
\end{prop}

\begin{prop}  \label{prop-local}
\textup{(\cite[Proposition~3.5]{BBGP})}
Let $\xi_0 \in \bdy \Theta$ 
and $B$ be a ball containing $\xi_0$.
Then
$\xi_0$ is regular with respect to $\Theta$
if and only if $\xi_0$ is regular with respect to $B \cap\Theta$.
\end{prop}

It is easy to see that regularity is invariant under translations, and we 
therefore formulate most of our regularity results around the origin.
See \cite{BBGP} for more on boundary regularity.

\section{Scaling invariance}
\label{sect-scaling}

The main aim of this section is to
 prove the following result.

\begin{prop} \label{prop-scaling}
Let $p \ne 2$, $a>0$ and $\Theta \subset \R^{n+1}$ be a domain
with $(0,0) \in \bdy \Theta$.
Set 
\[
    \Thetat=\{(ax,t)\in \R^{n+1} : (x,t) \in \Theta\}.
\]
Then $(0,0)$ is regular with respect to $\Theta$
if and only if it is regular with respect to $\Thetat$.
\end{prop}

A direct consequence, is that if 
$\theta:(-1,0) \to (0,\infty)$ is a bounded 
continuous function, then
$(0,0)$ is regular for 
$\partial_t u - \Delta_p u=0$, $p \ne 2$,  with respect to 
\[
     \Theta_K=\{(x,t)\in \R^{n+1}: |x| < K \theta(t) \text{ and } -1< t<0\}
\]
if and only if it is regular with respect to $\Theta_1$, i.e.\
regularity is independent of $K>0$.
Thus, there is no
Petrovski\u\i-type
criterion for $p \ne 2$ of the same type as for $p=2$.

\begin{proof}
Let $\ut$ be a function on $\Thetat$ and set
\[
    u(x,t)= K \ut(ax,t) \text{ for } (x,t) \in \Theta,
    \quad \text{where }
    K=a^{-p/(p-2)}.
\]
Then
\[
    \partial_t u(x,t)=K \partial_t \ut(ax,t)
    \quad \text{and} \quad
    \Delta_p u(x,t) = K^{p-1} a^p \Delta_p \ut(ax,t)
    = K \Delta_p \ut(ax,t),
\]
from which it follows that $u$ is \p-superparabolic in $\Theta$
if and only if $\ut$ is \p-super\-par\-a\-bol\-ic in $\Thetat$.
Next let $\ft \in C(\bdy \Thetat)$ and set 
\[
    f(x,t)= K \ft(ax,t) \quad \text{for } (x,t) \in \bdy\Theta.
\]
Then we see from the above that
\[
    \uPind{\Theta} f (x,t) = \uPind{\Thetat} (K\ft)(ax,t)
    \quad \text{for } (x,t) \in \Theta.
\]
This shows that regularity of the origin with respect to $\Theta$ 
implies regularity with respect to $\Thetat$.
The converse implication follows by switching the roles of $\Theta$ and
$\Thetat$, and replacing $a$  by $1/a$.
\end{proof}

We conclude this section by briefly comparing the linear and 
nonlinear cases regarding multiplied equations.

\begin{definition}
Let $1<p<\infty$.
A boundary point $\xi_0\in \partial \Theta$ is 
\emph{completely regular} with respect to $\Theta$, if
whenever  $f: \partial \Theta \to \R$ is continuous and $a>0$,
\[
        \lim_{\Theta \ni \xi \to \xi_0} \uPa f(\xi)=f(\xi_0)
\]
(where 
$\uPa$ denotes the upper Perron solution with respect
to the equation $a \partial_t u =\Delta_p u $),
i.e.\ whenever
$\xi_0$ is simultaneously regular for all the multiplied equations.
\end{definition}

\begin{remark}
\label{rmk-complete-reg}
By
Theorem~3.6 in Bj\"orn--Bj\"orn--Gianazza--Parviainen~\cite{BBGP}
regularity and 
complete regularity are the same when $p \ne 2$.
On the contrary, it follows from 
the classical Petrovski\u\i\ criterion that complete regularity
is a strictly stronger condition  when $p=2$.
The Petrovski\u\i\ criterion also shows that one may
replace ``regular'' by ``completely regular'' in 
Theorem~\ref{thm-Petrovskii} for $p=2$ as well,
thus providing examples of completely
regular boundary points for $p=2$.

More generally, consider
\[
\Theta=\{(x,t)\in\R^{n+1}: 
        |x|<\sqrt{-t}\sqrt{\log |{\log(-t)}|} h(t) \text{ and } -1< t<0\},
\]
where $h$ is a positive continuous function.
A scaling argument, similar to the proof of
Proposition~\ref{prop-scaling} (together with  Petrovski\u\i's criterion), 
shows that when $h(t):=K>0$ is constant, then 
$(0,0)$ is regular for $a  \partial_t u =\Delta u $ if
and only if $K \le 2/\sqrt{a}$.
Thus, for nonconstant $h$,
$(0,0)$ is \emph{completely} 
regular for $p=2$ if $\lim_{t \to 0\limminus} h(t)=0$,
while it is \emph{not} completely regular if $\liminf_{t \to 0\limminus} h(t) >0$.
Moreover, if $\lim_{t \to 0\limminus} h(t)=\infty$ then $(0,0)$
is not regular for any 
$a \partial_t u =\Delta u $.

The Petrovski\u\i\ criterion and the classical barrier characterization
for the heat equation show that the existence of a (traditional) barrier
for the heat equation does not imply complete regularity.
Lanconelli~\cite[Theorem~1.1]{Lanconelli77} showed
that if a point is regular
for $a_1  \partial_t u =\Delta u $ and $0 < a_2< a_1$, then
it is also regular for $a_2  \partial_t u =\Delta u $.
Thus the existence of a \emph{countable} barrier family with one barrier
for each $a=1,2,\ldots$, is equivalent to the complete regularity
when $p=2$.

All of this suggests that regularity for $p \ne 2$ rather corresponds
to complete regularity for $p=2$ than to regularity for $p=2$.
Also Proposition~\ref{prop-scaling} holds for $p=2$ and complete
regularity.

By Fabes--Garofalo--Lanconelli~\cite[Corollary~1.4]{FabesGarofaloLanconelli},
complete regularity for $p=2$ is equivalent to simultaneous regularity
for all linear 
parabolic equations 
of the form $ \partial_t u =\Div(A(x,t)\nabla u)$,
where $A(x,t)$ is a symmetric uniformly elliptic matrix
with $C^1$-Dini continuous coefficients.
\end{remark}

\section{The singular case \texorpdfstring{$1<p<2$}{}}
\label{sect-p<2}

We start this section by proving 
Theorem~\ref{thm-Petrovskii} in the singular range.

\begin{proof}[Proof of Theorem~\ref{thm-Petrovskii} for $1<p<2$.]
When $q>1/p$ and $K >1$, 
regularity was obtained  in
Bj\"orn--Bj\"orn--Gianazza--Parviainen~\cite[Proposition~7.1]{BBGP}. 
It follows
from Proposition~\ref{prop-restrict},
that any $K>0$ will do; this follows also from 
Proposition~\ref{prop-scaling}.

Now assume that $0 < q <1/p$.
By Proposition~\ref{prop-scaling} we can assume that $K=1$.
We shall construct an \emph{irregularity barrier}
(in the terminology of~\cite{KiLi96} 
and Petrovski\u{\i}~\cite[p.\ 389]{Petro2}).
Let 
\[
   u(x,t)= \begin{cases}
      \displaystyle \frac{|x|^{p/(p-1)}}{(-t)^{pq/(p-1)}} - \frac{n}{1-pq} 
       \biggl(\frac{p}{p-1}\biggr)^{p-1}(-t)^{1-pq}, &
      \text{if } (x,t) \in \overline{\Theta} \setm \{(0,0)\}, \\
      1, & \text{if } (x,t)= (0,0).
      \end{cases}
\]
Using Lemma~\ref{lem-Delta-powers} we see that in $\Theta$,
\[
    \partial_t u = \frac{pq}{p-1} \frac{|x|^{p/(p-1)}}{(-t)^{1+pq/(p-1)}}
    + \frac{n}{(-t)^{pq}}\biggl(\frac{p}{p-1}\biggr)^{p-1} 
    \ge \frac{n}{(-t)^{pq}} \biggl(\frac{p}{p-1}\biggr)^{p-1}
    = \Delta_p u.
\]
Hence, $\partial_t u - \Delta_p u \ge 0$ in $\Theta$,
which shows that $u$ is \p-superparabolic in $\Theta$.

Let $f=u|_{\bdy \Theta} \in C(\bdy \Theta)$
and let $v \in \LL_f(\Theta)$.
By the parabolic comparison principle (Theorem~\ref{thm-parabolic-comparison}),
with $T=0$, we see that
$v \le u$ in $\Theta$,
and thus
we also have $\lP f \le u$.
But then
\[
     \liminf_{\Theta \ni (x,t) \to (0,0)} \lP f(x,t) 
     \le     \liminf_{\Theta \ni (x,t) \to (0,0)} u(x,t) 
   \le    \liminf_{t \to 0 \limminus} u(0,t) 
    = 0 < 1 =f(0,0).
\]
Hence $(0,0)$ is irregular for $\Theta$.
\end{proof}

Next, we turn to Proposition~\ref{prop-one-barrier-intro}.
First, we formulate it in a different form
which also gives regularity for small boundary data.

\begin{prop} \label{prop-one-barrier}
Let $1<p<2$ and $0 < q  \le 1/p$.
Then there is a traditional barrier $u$ at $(0,0)$ for the domain
\[ 
\Theta=\{(x,t)\in\Rno: |x|<(-t)^q \text{ and } -1 < t<0\}.
\]
In particular if $f \in C(\bdy \Theta)$ satisfies
$|f-f(0,0)| \le g$ on $\bdy \Theta$,
where 
\[
g(x,t)= \frac{B}{2}\min\biggl\{ -t,\biggl( \frac{B}{2} \biggr)^{(p-1)/pq} \biggr\}^{1/(2-p)}
 \text{ and } 
B=\min\biggl\{n( 2-p)    \biggl(\frac{p}{p-1}\biggr)^{p-1},1\biggr\},
\]
then 
\begin{equation} \label{uP-small-f}
     \lim_{\Theta \ni (x,t) \to (0,0)} \uP f(x,t) 
     =       \lim_{\Theta \ni (x,t) \to (0,0)} \lP f(x,t)  =f(0,0).
\end{equation}
\end{prop}

Of course, we have a traditional barrier also when $q>1/p$.
The point here is that we obtain 
a traditional barrier
even at an irregular boundary point.
Note that for $q=1/p$ we find one traditional barrier, 
but we do not know whether the origin is regular or not.

\begin{proof}
Let 
\[
    v(x,t)=(-t)^ {1/(2-p)} (B-|x|^{p/(p-1)}),
      \quad (x,t) \in \Theta. 
\]
By Lemma~\ref{lem-Delta-powers}, we have in $\Theta$,
\[
    \Delta_p v = - n \biggl(\frac{p}{p-1}\biggr)^{p-1} (-t)^{(p-1)/(2-p)}
\]
and
\[
    \partial_t v = - \frac{1}{2-p} (-t)^{1/(2-p)-1} (B-|x|^{p/(p-1)})
     \ge        - \frac{B}{2-p} (-t)^{(p-1)/(2-p)}.
\]
Hence
\[
   \partial_t v - \Delta_p v \ge \biggl(n \biggl(\frac{p}{p-1}\biggr)^{p-1}
      - \frac{B}{2-p}\biggr) (-t)^{(p-1)/(2-p)}
   \ge0,
\]
and thus $v$ is \p-superparabolic in $\Theta$.
Next, let
\[
  \Thetap=\bigl\{(x,t)\in\Theta: |x|^{p/(p-1)} < \tfrac{1}{2} B\bigr\},
\]
and
\[
    M= \inf_{(x,t) \in \Theta \cap \bdy \Thetap} v(x,t) 
      = \biggl(\frac{B}{2}\biggr)^{1+(p-1)/pq(2-p)}.
\]
Then
\[
  u(x,t) = \begin{cases}
            \min\{v(x,t),M\}, & \text{if } (x,t) \in \Thetap, \\
            M, & \text{if } (x,t) \in \Theta \setm \Thetap,
            \end{cases}
\]
is \p-superparabolic in $\Theta$, by the pasting lemma
in Bj\"orn--Bj\"orn--Gianazza--Par\-vi\-ain\-en~\cite[Lemma~2.9]{BBGP}.
It is also easily seen that $u$ satisfies the remaining properties
required of a traditional barrier.

Finally, if $|f-f(0,0)| \le g$ on $\bdy \Theta$, then
for $(x_0,t_0) \in \bdy \Theta$,
\[
\limsup_{\Theta \ni (x,t) \to (x_0,t_0)} (f(0,0)-u(x,t)) \le f(x_0,t_0) \le
\liminf_{\Theta \ni (x,t) \to (x_0,t_0)} (f(0,0)+u(x,t))
\]
and hence
\[
   f(0,0)-u \le \lP f \le \uP f \le f(0,0)+u 
   \quad \text{in } \Theta,
\]
from which \eqref{uP-small-f} follows directly, as 
$\lim_{(x,t) \to (0,0)} u(x,t)=0$.
\end{proof}

\begin{proof}[Proof of Proposition~\ref{prop-one-barrier-intro}.]
The existence of a traditional barrier 
follows from Proposition~\ref{prop-one-barrier}
if $K=1$. For general $K>0$ the existence
follows from the scaling argument in the proof of
Proposition~\ref{prop-scaling}.
The irregularity is a direct consequence of
Theorem~\ref{thm-Petrovskii}.
\end{proof}

\section{The degenerate case \texorpdfstring{$p>2$}{}}
\label{sect-p>2}

The following theorem and its proof refine
the results in Lindqvist~\cite[Theorem, p.~571]{lindqvist95}
and Bj\"orn--Bj\"orn--Gianazza--Parviainen~\cite[Theorem~6.1]{BBGP}.

\begin{thm}\label{thm:Petr-deg}
Let $p>2$, $t_0 <0$ and
\[
\Theta=\{(x,t)\in\R^{n+1}: 
     |x|   <\zeta(t) \text{ and } t_0<t<0 \},
\]
where 
$\zeta$  is a positive continuous 
function on $(t_0,0)$ such that 
\begin{equation}  \label{eq-lim-de-new}
\lim_{t\to0\limminus}(-t)^{-1/p}\zeta(t)=0.
\end{equation} 
Then the origin $(0,0)$ is regular with respect to $\Theta$.
\end{thm}

In the converse direction we have the following result,
which shows that Theorem~\ref{thm:Petr-deg} is essentially sharp.
That $\zeta=(-t)^{1/p}$ with $p>2$
implies irregularity was mentioned as a footnote 
already in~\cite[p.~572]{lindqvist95}, with no further details.
Here we strengthen the statement and
provide a full proof of the result.

\begin{prop}\label{Prop:Petr-deg-irr}
Let $p>2$, $t_0 <0$ and 
\[
\Theta=\{(x,t)\in\R^{n+1}: 
        |x|   <\zeta(t) \text{ and } t_0<t<0 \},
\]
where 
$\zeta$ is a positive continuous 
function on $(t_0,0)$ 
such that 
\[
\liminf_{t\to0\limminus}(-t)^{-1/p}\zeta(t)>0.
\]
Then the origin $(0,0)$ is irregular with respect to $\Theta$.
Moreover, there is no traditional barrier at $(0,0)$.
\end{prop}

As an irregular borderline case one might at first think that 
this could provide a counterexample showing that
our conjecture after Proposition~\ref{prop-one-barrier-intro}
is true.
However, the last part of Proposition~\ref{Prop:Petr-deg-irr}
shows that this is not possible in this case.

\begin{proof}[Proof of Theorem~\ref{thm-Petrovskii} for $p>2$.]
This follows directly from Theorem~\ref{thm:Petr-deg} and
Proposition~\ref{Prop:Petr-deg-irr}.
\end{proof}

\begin{proof}[Proof of Proposition~\ref{Prop:Petr-deg-irr}]
By assumption there is $m>0$ and $t_1$ such that $t_0 \le t_1<0$ and 
\[
(-t)^{-1/p}\zeta(t)> m
 \quad  \text{for } t_1 < t < 0.
\] 
Let 
\[
\Thetap=\{(x,t)\in\R^{n+1}: 
    |x|    <m (-t)^{1/p} \text{ and } t_1<t<0
 \} \subset \Theta.
\]
If we show that $(0,0)$ is irregular with respect to $\Thetap$,
then by Proposition~\ref{prop-restrict},
$(0,0)$ is irregular with respect to $\Theta$  as well.
By Proposition~\ref{prop-scaling} we may assume that $m=1$,
and by Proposition~\ref{prop-local} we may assume that $t_1=-1$.

As in Section~\ref{sect-p<2}, we construct an irregularity barrier.
Let 
\[
   u(x,t)=C \biggl(\frac{|x|^p}{-t}\biggr)^{1/(p-2)},
   \quad (x,t) \in \clThetap \setm \{(0,0)\},
\]
where $C$ is a positive constant that will be determined later.
Lemma~\ref{lem-Delta-powers} with $\al=p/(p-2)$ shows that
\[
   \Delta_p u  = \biggl(\frac{C\al}{(-t)^{1/(p-2)}}\biggr)^{p-1} (n+\alp) |x|^\alp 
   \quad \text{and} \quad
  \partial_t u  = \frac{C}{p-2} \frac{|x|^\al}{(-t)^{(p-1)/(p-2)}}.
\]
Thus, it follows that in $\Thetap$ we have
\begin{align*}
\partial_t u - \Delta_p u & =
    \frac{C|x|^\al}{(-t)^{(p-1)/(p-2)}}
    \biggl( \frac{1}{p-2} - C^{p-2}(n+\alp) \al^{p-1}\biggr) 
    \ge 0,
\end{align*}
provided that
\[
0<C^{p-2} \le \frac{1}{(p-2)(n+\alp)\al^{p-1}}
= \frac{(p-2)^{p-1}}{\la p^{p-1}},
\]
where $\la=n(p-2)+p=(p-2)(n+\alp)$. 
This makes $u$ into a positive \p-superparabolic function in $\Thetap$.
Next, it is easy to see that $f:\bdy \Thetap \to \R$ given by
\[
    f(x,t)=\begin{cases}
       u(x,t), & \text{if } (x,t) \in \bdy \Thetap \setm \{(0,0)\}, \\
       C, & \text{if } (x,t)= (0,0)
      \end{cases}
\]
is continuous.

Now, let $v \in \LL_f(\Thetap)$.
By the parabolic comparison principle (Theorem~\ref{thm-parabolic-comparison}),
with $T=0$, we see that
$v \le u$ in $\Thetap$,
and thus
we also have $\lPind{\Thetap} f \le u$.
But then
\[
     \liminf_{\Thetap \ni (x,t) \to (0,0)} \lPind{\Thetap} f(x,t) \le 
   \liminf_{\Thetap \ni (x,t) \to (0,0)} u(x,t) =0 < C=f(0,0),
\]
as $u(0,t)=0$ for $t_1<t<0$.
Hence, $(0,0)$ is irregular for $\Thetap$ and thus for $\Theta$.

Next we turn to the existence of a traditional barrier.
As in the beginning of the proof, we can reduce to $\Thetap$ with $m=1$ here as well;
if $\Theta$ had a traditional barrier
at the origin, then its restriction
to $\Thetap$ would be a traditional barrier, and after scaling we would
have a traditional barrier with $m=1$.

Assume that $w$ is a traditional barrier at $(0,0)$ for $\Thetap$ with $m=1$.
Extending $w$ to $\bdy \Thetap$ by letting
\[
     w(\xi_0)=\liminf_{\Thetap \ni \xi \to \xi_0} w(\xi)
     \quad\text{for all }\xi_0\in\partial\Thetap,
\]
makes $w$ into a lower semicontinuous function on $\clThetap$.
Moreover, $w$ is positive in $\clThetap \setm \{(0,0)\}$ and 
 $w(0,0)=0$.
Let 
\[ 
    C=\min\biggl\{\min_{(x,t_1) \in \bdy \Thetap} w(x,t_1), 
       \biggl(\frac{(p-2)^{p-1}}{\la p^{p-1}}\biggr)^{1/(p-2)}\biggr\},
\]
and let $u>0$ be the \p-superparabolic irregularity barrier 
constructed above with this (admissible) $C$.
Then $C-u$ is a \p-subparabolic function in $\Thetap$ such that
\[
     \limsup_{\Thetap \ni (x,t) \to (x_0,t_0)} (C-u(x,t)) \le
     \liminf_{\Thetap \ni (x,t) \to (x_0,t_0)} w(x,t) 
\]
for all $(x_0,t_0) \in \bdy \Thetap \setm \{(0,0)\}$.
Hence, by the parabolic comparison principle (Theorem~\ref{thm-parabolic-comparison})
again, $C-u \le w$ in $\Thetap$, and thus
\[
   \limsup_{\Thetap \ni (x,t) \to (0,0)} w(x,t) 
   \ge    
   C-     \liminf_{\Thetap \ni (x,t) \to (0,0)} u(x,t) 
   = C,
\]
which contradicts the fact that $w$ is a traditional barrier.
Hence, there is no traditional barrier at $(0,0)$.
\end{proof}

\begin{proof}[Proof of Theorem~\ref{thm:Petr-deg}]
It will be convenient to rewrite $\Theta$ as
\[
\Theta=\biggl\{(x,t)\in\R^{n+1}: 
 \biggl(\frac{|x|}{(-t)^{1/\lambda}}\biggr)^{p/(p-1)}
   <\de(t)
 \text{ and } t_0<t<0 \biggr\},
\]
where 
\[
   \de(t)=\biggl(\frac{\zeta(t)}{(-t)^{1/\la}}\biggr)^{p/(p-1)}.
\] 
Let 
\[
     \be = \frac{n(p-2)}{\la} 
     \quad \text{and} \quad
     \ga = \frac{n(p-2)}{\la(p-1)} = \frac{\be}{p-1} < \be.
\]
Then it follows directly that 
\eqref{eq-lim-de-new} is equivalent to 
\begin{equation}  \label{eq-lim-de}
\lim_{t\to0\limminus}(-t)^{-\ga}\de(t)=0.
\end{equation} 
In the proof we will use
two additional properties of the function $\de$,
namely that 
\begin{equation} \label{eq-wlog}
  \de \text{ is smooth} \quad \text{and} \quad
 t \mapsto (-t)^{-\be}\de(t) \text{ is nondecreasing}.
\end{equation}
First,  let us 
show how we can assume this without loss of generality.
Let 
\[
    h(t)=(-t)^{-\be}\de(t),
    \ 
    \tilde{h}(t)=\sup_{t_0 < s \le t} h(s)
    \text{ and }
    \td(t)=(-t)^\be \tilde{h}(t)
    \quad
    \text{for } t_0 < t < 0.
\]
Then $\td\ge\de$ and $(-t)^{-\be}\td(t)=\tilde{h}(t)$ is nondecreasing.
We also need that
\begin{equation} \label{eq-lim=0}
    0=\lim_{t\to0\limminus} (-t)^{-\ga} \td(t)
     = \lim_{t\to0\limminus} (-t)^{\be -\ga} \tilde{h}(t).
\end{equation}
Assume that this is false.
Then there is $\eps>0$ and $t_j \nearrow 0$ so that
$
    (-t_j)^{\be - \ga} \tilde{h}(t_j) > \eps
$
for $j=1,2,\ldots$\,.
As $\be - \ga >0$, we have $\limsup_{j\to\infty}\tilde{h}(t_j)=\infty$,
and so we can for each $j$ find a $k_j>j$ such that
$\tilde{h}(t_{k_j}) > \tilde{h}(t_j)$.
By the definition of $\tilde{h}$,
and the continuity of $h$,  there is some $s_j$ such that
\[
   t_j < s_j \le t_{k_j}
   \quad \text{and} \quad
   \tilde{h}(t_{k_j})=h(s_j).
\]
Thus
\[
  (-s_j)^{-\ga} \de(s_j)
   = (-s_j)^{\be - \ga} h(s_j) 
   \ge
    (-t_{k_j})^{\be - \ga} \tilde{h}(t_{k_j})
   > \eps.
\]
But this contradicts \eqref{eq-lim-de}, 
and  hence \eqref{eq-lim=0} is true.
Finally we can find a smooth $\hat{\de}$ such that
$\td < \hat{\de}  < 2\td$ and $(-t)^{-\be} \hat{\de}(t)$ is
nondecreasing.
Note that 
\[
  \lim_{t\to0\limminus}(-t)^{-\ga}\hat{\de}(t)=0.
\]
If we define $\Thetah \supset \Theta$ in the same way as 
$\Theta$, but using $\hat{\de}$
instead of $\de$, and $\Thetah$ is regular, then also $\Theta$
is regular, by 
Proposition~\ref{prop-restrict}.
We have thus shown that we may assume \eqref{eq-wlog} without
loss of generality.

By 
Theorem~\ref{thm:barrier-char}, 
it is enough to show that there exists a  barrier family 
$\{w_C\}_{C=C_0}^\infty$ in $\Theta$  at the origin $\xi_0=(0,0)$.
The  family $\{w_C\}_{C=C_0}^\infty$  we construct will be smooth in $\Theta$, 
so that
${\partial_t w_C}-\Delta_p w_C\ge0$ is satisfied in the classical sense.
It will be constructed in the form
\[
w_C(x,t) = (Q(x,t)^{(p-1)/(p-2)} - C^{(p-1)/(p-2)}) f(t)+ \rho_C(t),
\]
where $C>0$ and
\begin{align*}
Q(x,t) &=C+\frac{p-2}{p\lm^{1/(p-1)}}
    \biggl(\frac{|x|}{(-t)^{1/\lm}}\biggr)^{p/(p-1)}, \\
 f(t) & = -\de(t)^{1/(p-2)} (-t)^{-n/\la}, \nonumber \\
 \rho_C(t) &  = -C^{1/(p-2)} \de(t) f(t). \nonumber 
\end{align*}
Note that $f<0$ and $f$ is a smooth nonincreasing function, 
by assumption \eqref{eq-wlog}.

We shall show that $w_C$ is 
a positive supersolution 
in $\Theta$ if $C$ is large enough. In the calculations below, 
we will for simplicity drop the subscript  $C$ in $w_C$ and $\rho_C$.
We shall also often omit the arguments 
and only write $w$, $Q$, $f$ and  $\rho$.
Note that $w$ is positive when
\begin{equation}\label{Eq:petr:2}
Q(x,t)^{(p-1)/(p-2)}-C^{(p-1)/(p-2)}<-\frac{\rho(t)}{f(t)} = C^{1/(p-2)} \de(t).
\end{equation}
Moreover, since $Q\ge C$ and $f<0$, we have by 
assumption~\eqref{eq-lim-de} that
\[
w(x,t)
 \le \rho(t) = -C^{1/(p-2)} \de(t) f(t)  
= C^{1/(p-2)} \de(t)^{(p-1)/(p-2)} (-t)^{-n/\la} \to0, 
\]
as $t\to0\limminus$. 
Thus, \ref{cond-third} in Definition~\ref{def-barrier} holds.

In order to prove \ref{cond-first} in Definition~\ref{def-barrier}, 
we need to show that the domain defined 
by \eqref{Eq:petr:2} contains $\Theta$. 
Indeed, in $\Theta$ we have 
\[
\biggl( \frac{|x|}{(-t)^{1/\lm}} \biggr)^{p/(p-1)}< \de(t).
\]
The elementary inequality $(1+s)^{\al} < 1+\al s(1+s)^{\al-1}$ with 
$\al=(p-1)/(p-2)>1$
then yields that for sufficiently large $C$
we have in $\Theta$, 
\begin{align}  \label{eq-elem-ineq}
Q^{(p-1)/(p-2)} - C^{(p-1)/(p-2)} 
&\le 
  C^{(p-1)/(p-2)} \biggl[ \biggl(1+\frac{(p-2)\de(t)}{pC\lm^{1/(p-1)}}
        \biggr)^{(p-1)/(p-2)} - 1\biggr] \nonumber\\
&\le C^{(p-1)/(p-2)} \frac{(p-1)\de(t)}{pC\lm^{1/(p-1)}}
      \biggl(1+\frac{(p-2)\de(t)}{pC\lm^{1/(p-1)}} \biggr)^{1/(p-2)} 
   \nonumber \\
&\le \frac{p-1}{p} 
     C^{1/(p-2)} \de(t), 
\end{align}
since $\la >1$.
Thus $\Theta$ is contained in the domain defined by \eqref{Eq:petr:2}
if $C$ is large enough.

Next we
show that $w$ is \p-superparabolic in $\Theta$.
We have
\begin{align*}
\nabla Q& =\frac{p-2}{(p-1)\lm^{1/(p-1)}}
    \frac{|x|^{(2-p)/(p-1)}x}{(-t)^{p/\lm(p-1)}},\\
\partial_t Q& =\frac{p-2}{(p-1)\lm^{p/(p-1)}}
      \biggl(\frac{|x|}{(-t)^{1/\lm}}\biggr)^{p/(p-1)}\frac1{-t}
   =\frac{p (Q-C)}{\lm(p-1)(-t)}.
\end{align*}
Since $w(x,t)=(Q^{(p-1)/(p-2)}-C^{(p-1)/(p-2)})f(t)+\rho(t)$, we have
\begin{align*}
\nabla w &= \frac{p-1}{p-2} f Q^{1/(p-2)}\nabla Q 
=\frac{Q^{1/(p-2)}f}{\lm^{1/(p-1)}} \frac{|x|^{(2-p)/(p-1)}x}{(-t)^{p/\lm(p-1)}},\\
|\nabla w|^{p-2} \nabla w
    &=\frac{Q^{(p-1)/(p-2)}|f|^{p-2}f}{\lm}\frac{x}{(-t)^{p/\lm}}.
\end{align*}
Therefore, 
\begin{align*}
\Delta_p w &=\frac{Q^{(p-1)/(p-2)}|f|^{p-2}f}{\lm}\frac{n}{(-t)^{p/\lm}}\\
       & \quad + \frac{Q^{1/(p-2)}|f|^{p-2}f}{\lm^{p/(p-1)}}
    \biggl(\frac{|x|}{(-t)^{1/\lm}}\biggr)^{p/(p-1)}\frac1{(-t)^{p/\lm}}\\
&\le \frac{Q^{(p-1)/(p-2)}}{\lm} \frac{n}{(-t)^{p/\lm}} |f|^{p-2}f,
\end{align*}
since $f<0$. Moreover,
\begin{align*}
\partial_t w &= (Q^{(p-1)/(p-2)}-C^{(p-1)/(p-2)})f' + \rho'
      +\frac{p-1}{p-2}f Q^{1/(p-2)}\partial_t Q\\
   &= (Q^{(p-1)/(p-2)}-C^{(p-1)/(p-2)})f' + \rho'
     +\frac{p fQ^{1/(p-2)}(Q-C)}{\lm(p-2)(-t)}.
\end{align*}
Combining the previous expressions yields
\begin{align*}
\partial_t w-\Delta_p w
&\ge \rho' + (Q^{(p-1)/(p-2)}-C^{(p-1)/(p-2)})f' \\
&\quad +\frac{p fQ^{1/(p-2)}(Q-C)}{\lm(p-2)(-t)}
- \frac{Q^{(p-1)/(p-2)}}{\lm} \frac{n}{(-t)^{p/\lm}} |f|^{p-2}f.
\end{align*}
In the domain given by \eqref{Eq:petr:2} (and thus in $\Theta$) we have 
\[
  0\le Q^{(p-1)/(p-2)}-C^{(p-1)/(p-2)} < -\frac{\rho(t)}{f(t)},
\]
and as $f'\le0$, this yields
\begin{align*}
\partial_t w-\Delta_p w
&\ge \rho' -\frac{\rho f'}{f}  
+\frac{p fQ^{1/(p-2)}(Q-C)}{\lm(p-2)(-t)}
- \frac{Q^{(p-1)/(p-2)}}{\lm} \frac{n}{(-t)^{p/\lm}} |f|^{p-2}f \\
&= \biggl( \biggl( \frac{\rho}{f} \biggr)' +\frac{p Q^{1/(p-2)}(Q-C)}{\lm(p-2)(-t)} 
- \frac{nQ^{(p-1)/(p-2)}}{\lm} \frac{|f|^{p-2}}{(-t)^{p/\lm}} \biggr)f(t)\\
&=:H(x,t) f(t).
\end{align*}
Since $f<0$, the last expression will be nonnegative if $H(x,t)\le0$.
We have $\rho(t)/f(t) =-C^{1/(p-2)}\de(t)$, and
in $\Theta$ we have for sufficiently large $C$,
\[
C\le Q \le C + \frac{(p-2)\de(t)}{p\lm^{1/(p-1)}} \le  2C. 
\]
This yields 
\begin{align*}
H(x,t) &\le -C^{1/(p-2)}\de'(t) 
+\frac{Q^{1/(p-2)}\de(t)}{\lm^{p/(p-1)}(-t)} 
- \frac{nQ^{(p-1)/(p-2)}}{\lm} \frac{|f|^{p-2}}{(-t)^{p/\lm}}\\
&\le -C^{1/(p-2)}\de'(t) 
+\frac{(2 
          C)^{1/(p-2)}\de(t)}{\lm^{p/(p-1)}(-t)} 
- \frac{nC^{(p-1)/(p-2)}}{\lm} \frac{|f|^{p-2}}{(-t)^{p/\lm}}.
\end{align*}
Now, 
\[
\frac{|f|^{p-2}}{(-t)^{p/\lm}} = \frac{\de(t) (-t)^{-n(p-2)/\la}}{(-t)^{p/\la}} 
= \frac{\de(t)}{-t},
\]
from which it follows that for sufficiently large $C$,
\begin{align*}
H(x,t) &\le C^{1/(p-2)} \biggl[ - \de'(t) 
+ \biggl(\frac{2 
       ^{1/(p-2)}}{\lm^{p/(p-1)}} 
- \frac{nC}{\lm} \biggr) \frac{\de(t)}{-t} \biggr] \\
&\le C^{1/(p-2)} \biggl[ - \de'(t) - \be \frac{\de(t)}{-t} \biggr] \le 0,
\end{align*}
since $(-t)^{-\be}\de(t)$ is nondecreasing by \eqref{eq-wlog}.
Thus we have shown that $w_C$ is a supersolution in $\Theta$,
if $C$ is large enough.

Finally, we show that \ref{cond-second-weak} in Definition~\ref{def-barrier}
is satisfied. 
For $(x,t) \in \overline{\Theta}\setm \{(0,0)\}$ we have 
\[ 
  \biggl(\frac{|x|}{(-t)^{1/\lambda}}\biggr)^{p/(p-1)} \le \de(t),
\]
and hence by \eqref{eq-elem-ineq},
\[
Q(x,t)^{(p-1)/(p-2)} - C^{(p-1)/(p-2)} \le \frac{p-1}{p} 
    C^{1/(p-2)} \de(t).
\]
Since $f<0$, this implies that
\begin{align*}
w_C(x,t) &\ge \frac{p-1}{p} 
       C^{1/(p-2)} \de(t) f(t) - C^{1/(p-2)} \de(t) f(t)\\
&= -\frac{1}{p} 
      C^{1/(p-2)} \de(t) f(t)
= \frac{1}{p} 
    C^{1/(p-2)} \de(t)^{(p-1)/(p-2)} (-t)^{-n/\la}. 
\end{align*}
From~\eqref{eq-wlog} we conclude that
$(-t)^{-\be}\de(t)\ge\theta$ for $t_0/2<t<0$ and some $\theta>0$.
Hence, 
\[
w_C(x,t)\ge \frac1p C^{1/(p-2)} \theta^{(p-1)/(p-2)} (-t)^{(p-2)n/\la} >0
\]
for those $t$.

As $(x,t)\in\partial\Theta$ with $|(x,t)|\ge 1/k$ implies that 
$-t\ge \eps_k$ for some $\eps_k>0$, this shows
that $\{w_C\}_{C=C_0}^\infty$ is a  barrier family for the domain
$\Theta^*=\{(x,t)\in\Theta:t>t_0/2\}$,
provided that $C_0$ is large enough.
It thus follows from Theorem~\ref{thm:barrier-char} that
$(0,0)$ is regular
with respect to $\Theta^*$ and thus with respect to $\Theta$, by
Proposition~\ref{prop-local}.
\end{proof}

Even though the domain $\Theta$ in \eqref{eq-Theta-partial-reg}
below is irregular (by Theorem~\ref{thm-Petrovskii})
and does not have a traditional
barrier at the origin, we can still
obtain regularity for some \emph{small} functions 
vanishing
at $(0,0)$ as well as at $(0,-1)$.

\begin{prop} 
Let $0 < q \le 1/p$,
\begin{equation} \label{eq-Theta-partial-reg}
   \Theta = \{(x,t) \in \R^{n+1}  : |x| < (-t) ^q \text{ and } -1 < t < 0\}.
\end{equation}
and
\[
     u(x,t)= \begin{cases}
           \displaystyle A \biggl(\frac{|x|^p}{(-t)^\be}\biggr)^{1/(p-2)},
             & \text{if } (x,t) \in \overline{\Theta} \setm \{(0,0)\}, \\
             0, & \text{if } (x,t)= (0,0),
             \end{cases}
\] 
where
\[
    0 < \be < pq\le1
    \quad \text{and} \quad
    A= \biggl(\frac{\be}{\la} \biggl(1-\frac{2}{p}\biggr)^{p-1}
       \biggr)^{1/(p-2)}.
\]
If $f \in C(\bdy \Theta)$ satisfies
\[
    |f(x,t)-f(0,0)| \le u(x,t) 
    \quad \text{for } (x,t) \in \bdy \Theta,
\]
then
\begin{equation} \label{eq-f-reg}
     \lim_{\Theta\ni (x,t) \to (0,0)} \lP f (x,t) 
     = \lim_{\Theta\ni (x,t) \to (0,0)} \uP f (x,t) 
     = f(0,0).
\end{equation}
\end{prop}

It is easy  to see that it is preferable to choose $\be$ as large as possible.
However, we cannot choose $\be= pq$ as then $u$ would not be continuous
at the origin.

The function $u$ constructed above fails to be a traditional barrier only in
one respect, namely  $\lim_{\Theta\ni (x,t) \to (0,-1)} u(x,t)= u(0,-1)=0$.
Thus, one requirement on $f$ is that
$f(0,-1)=f(0,0)$. Moreover, it also follows from the proof below that
\[
     \lim_{\Theta\ni (x,t) \to (0,-1)} \lP f (x,t) 
     = \lim_{\Theta\ni (x,t) \to (0,-1)} \uP f (x,t) 
     = f(0,-1) = f(0,0).
\]
Obviously one can obtain similar results for
\[
   \Theta = \{(x,t) \in \R^{n+1}  : |x| < K (-t) ^q \text{ and } t_0 < t < 0\}.
\]
when $K>0$ and $t_0 <0$.

\begin{proof}
By Lemma~\ref{lem-Delta-powers} with $\al=p/(p-2)$ we have
\[
\Delta_p u = \biggl( \frac{A\al}{(-t)^{\be/(p-2)}} \biggr)^{p-1} 
             \frac{\la}{p-2} |x|^\al 
\]
and
\[
  \partial_t u = \frac{A \be}{(p-2)(-t)^{\be/(p-2)+1}} |x|^\al .
\]
Thus, in $\Theta$, 
\begin{align*}
   \partial_t u - \Delta_p u & = \frac{A |x|^\al}{(p-2)(-t)^{\be/(p-2)+1}}
     ( \be - A^{p-2}\al^{p-1} \la (-t)^{1-\be} )\\
   & \ge \frac{A |x|^\al}{(p-2)(-t)^{\be/(p-2)+1}}
     ( \be - A^{p-2}\al^{p-1} \la )
   = 0,
\end{align*}
where we have used that $\be<1$.
Hence, $u$ is \p-superparabolic in $\Theta$. 
Moreover, as $\be<pq$, we see that $u \in C(\overline{\Theta})$.
So $u + f(0,0) \in \UU_f$ and $-u +f(0,0) \in \LL_f$.
Hence
\[
  f(0,0) = \lim_{\Theta\ni (x,t) \to (0,0)} (f(0,0)-u(x,t)) 
  \le \liminf_{\Theta\ni (x,t) \to (0,0)} \lP f(x,t)
\]
and
\[ 
   \limsup_{\Theta\ni (x,t) \to (0,0)} \uP f(x,t)
  \le \lim_{\Theta\ni (x,t) \to (0,0)} (f(0,0)+u(x,t)) 
  = f(0,0),
\]
which together with the inequality $\lP f \le \uP f$ yield
\eqref{eq-f-reg}.
\end{proof}


\end{document}